\documentclass[11pt]{article}

\usepackage{fullpage}

\usepackage{amsmath,amssymb,amsthm}
\usepackage{amsfonts,latexsym}
\usepackage{enumitem}
\usepackage{mathrsfs}
\usepackage{verbatim}

\newtheorem{thm}{Theorem}[section]
\newtheorem{Proposition}[thm]{Proposition}

\newtheorem*{Lemma*}{Lemma}

\numberwithin{equation}{section}

\newcommand{\Us}{\mathrm{u}} 
\newcommand{\Ds}{\mathrm{d}}
\newcommand{\Hs}{\mathrm{h}}

\newcommand{\D}{\mathcal{D}} 
\newcommand{\E}{\mathcal{E}} 
\newcommand{\G}{\mathcal{G}} 
 
\newcommand{\A}{\mathcal{A}} 
\newcommand{\B}{\mathcal{B}} 
\newcommand{\K}{\mathcal{K}} 
\newcommand{\Ll}{\mathcal{L}} 

\title{Equivalence classes of ballot paths modulo strings of length 2 and 3}
\author{K. Manes, A. Sapounakis, I. Tasoulas and P. Tsikouras}
\date{}
\begin{document}
\maketitle

\begin{abstract}
Two paths are equivalent modulo a given string $\tau$,
whenever they have the same length and the positions 
of the occurrences of $\tau$ are the same in both paths.
This equivalence relation was introduced 
for Dyck paths in \cite{BP},
where the number of equivalence classes
was evaluated for any string of length 2.

In this paper, we evaluate the number
of equivalence classes in the set of ballot paths
for any string of length 2 and 3,
as well as in the set of Dyck paths for any string 
of length 3.
\end{abstract}

\section{Introduction}

Throughout this paper, a \emph{path} is considered to be a lattice path on the integer plane, 
consisting of steps $\Us = (1,1)$ (called \emph{rises}) and $\Ds = (1,-1)$ (called \emph{falls}).
Since the sequence of steps of a path is encoded by a word in $\{ \Us, \Ds \}^*$,
we will make no distinction between these two notions.
The \emph{length} $|\alpha|$ of a path $\alpha$ is the number of its steps.
The \emph{height of a point} of a path is its $y$-coordinate.
For any path $\alpha$ and $i  \geq 1$,
we define $\alpha^i = \alpha \alpha \cdots \alpha$ ($i$ times), and $\alpha^0 = \varepsilon$,
where $\varepsilon$ is the empty path.

A \emph{Dyck path} is a path that starts and ends at the same height and lies weakly above this height.
It is convenient to consider that the starting point of a Dyck path is the origin of a pair of axes.

The set of Dyck paths of semilength $n$ is denoted by $\mathcal{D}_n$, and we set
$\mathcal{D} = \bigcup_{n \geq 0} \mathcal{D}_n$, where $\mathcal{D}_0 = \{ \varepsilon \}$.
It is well known that $|\mathcal{D}_n| = C_n$, where $C_n = \frac{1}{n+1} \binom{2n}{n}$
is the $n$-th Catalan number (sequence A000108 in \cite{Sloane}). 
The generating function of the Catalan numbers is denoted by $C(x)$, where  
\[
C(x) = 1+ xC^2(x) = \frac{1-\sqrt{1-4x}}{2x} = \sum_{n \geq 0} C_n x^n.
\]

We will also use the generating function $M(x)$ of the Motzkin numbers (sequence A001006 in \cite{Sloane}), where
\[
M(x) = 1 + xM(x) + x^2M^2(x) = \frac{1-x-\sqrt{1-2x+3x^2}}{2x^2}.
\]


A path which is a prefix of a Dyck path, is called \emph{ballot path}.
The set of ballot paths of length $n$ is denoted by $\mathcal{P}_n$ and 
$\mathcal{P} = \bigcup_{n \geq 0} \mathcal{P}_n$.

A path $\tau \in \{ \Us, \Ds \}^*$, called in this context \emph{string}, \emph{occurs} in a path $\alpha$ if $\alpha = \beta \tau \gamma$, for some $\beta, \gamma \in \{ \Us, \Ds \}^*$.
The number of occurrences of the string $\tau$ in $\alpha$, is denoted by $|\alpha|_\tau$.
In particular, $|\alpha|_\Us$ and $|\alpha|_\Ds$ are the number of rises and falls of $\alpha$ respectively.  
We also use the notation $h(\alpha) = |\alpha|_\Us - |\alpha|_\Ds$, which corresponds to the height 
of the last point of $\alpha$, whenever $\alpha$ starts at height zero.

We say that an occurrence of the string $\tau$ in the path $\alpha$ is at height 
$j \geq 0$,
whenever the minimum height of the points of $\tau$ in this occurrence is equal to $j$.
Furthermore, we say that an occurrence of the string $\tau$ in 
a path $\alpha$ of length $n$ is at position $i \in [n] = \{1,2, \ldots, n\}$,
whenever the first step of this occurrence is the $i$-th step of $\alpha$.

Many articles dealing with the number of occurrences of strings in Dyck paths 
have appeared in the literature
(e.g., \cite{D3, M1, M2, MSV, STT1}). 
More general results on this subject are given in \cite{MSTT1, MSTT2}.

In another direction, taking into account not only the number of occurrences of
the string $\tau$ in a path, but also the positions of these occurrences,
we consider the equivalence relation $\underset{\tau}{\sim}$ on the set of paths
defined by
\[
\alpha \underset{\tau}{\sim} \alpha^\prime  \text{ iff } 
|\alpha| = |\alpha^\prime|
\text{ and the positions of the occurrences of $\tau$ in $\alpha$ and 
$\alpha^\prime$ are the same}.
\]

Recently, Baril and Petrossian in \cite{BP} and \cite{BP2}, 
introduced this equivalence relation on the set of Dyck paths 
and the set of Motzkin paths respectively,
and they evaluated the number of $\tau$-equivalence classes
(i.e., the classes with respect to $\underset{\tau}{\sim}$)
in terms of generating functions for any string $\tau$ of length at most 2.

In this paper, we examine this equivalence relation on the set of ballot paths 
and we evaluate the number of $\tau$-equivalence classes of $\mathcal{P}_n$,
for every string $\tau$ of length 2 or 3,
as well as the number of \emph{Dyck classes} 
(i.e., classes containing at least one Dyck path),
for every string $\tau$ of length 3.

Firstly, for each case of $\tau$, we decompose the paths $p \in \mathcal{P}$ as 
\begin{equation}\label{clusters}
p = a_0 c_1 a_1 c_2 a_2 \cdots c_k a_k, \qquad k \geq 0,
\end{equation}
where the $c_i$'s are the maximal clusters of $\tau$ in $p$ and 
the $a_i$'s are the \emph{$\tau$-components}, 
i.e., the subpaths of $p$ avoiding $\tau$ and lying between these clusters. 
We also use the notation $a(p)=a_k$.
Clearly two ballot paths $p = a_0 c_1 a_1 c_2 a_2 \cdots c_k a_k$ and 
$p^\prime = a_0^\prime c_1^\prime a_1^\prime c_2^\prime a_2^\prime \cdots c_{k^\prime}^\prime a_{k^\prime}^\prime$
are $\tau$-equivalent iff $k = k^\prime$ and $|a_i|= |a_i^\prime|$, $c_i = c_i^\prime$,
for all $i\leq k$.

Next, by assuming several conditions for the $\tau$-components,
we define a set $\A^{(\tau)}_n$ of representatives of length $n$, 
where $\A^{(\tau)} = \bigcup_{n\geq 0} \A_n^{(\tau)}$ and $\A_0^{(\tau)} = \{\varepsilon\}$,
and we prove the following result:
\begin{Proposition}\label{p1}
For every $p \in \mathcal{P}_n$, there exists a unique $p^\prime \in \A^{(\tau)}_n$,
such that $p \underset{\tau}{\sim} p^\prime$, $n \geq 0$.
\end{Proposition}

Finally, by enumerating the set of representatives, 
we evaluate the number of $\tau$-equivalence classes in both cases of ballot and Dyck paths.
As a result, in several cases we obtain some well known sequences,
e.g., the Fibonacci numbers $f_n$ (see sequence A000045 in \cite{Sloane}), 
whereas in other cases we introduce some new ones.

Throughout this paper, we denote by $h_i$, $i \in [k]$, the height of the last point of $c_i$
and we define $h_0=0$ and $h_{k+1}$ to be the height of the last point of $a_k$. 
We can easily check that, for any string $\tau$, if
$p, p^\prime \in \mathcal{P}$, such that $p \sim p^\prime$, and
$p, p^\prime$ are decomposed according to \eqref{clusters}, i.e.,
\[
p = a_0 c_1 a_1 \cdots c_k a_k, \qquad 
p^\prime = a_0^\prime c_1 a_1^\prime \cdots c_k a_k^\prime,
\]
then
\begin{align}\label{parity}
h_{i+1} - h_{i+1}^\prime &= h_i - h_i^\prime +2(|a_i|_\Us - |a_i^\prime|_\Us),
\end{align}
which gives that 
$h_{i}$ and $h_i^\prime$, have the same parity, $0 \leq i \leq k+1$.

Whenever there is no danger of confusion, 
we will write for simplicity $\sim$, $\A$, ``component'' and ``equivalent'' 
instead of $\underset{\tau}{\sim}$, $\A^{(\tau)}$, ``$\tau$-component'' 
and ``$\tau$-equivalent'' respectively.
All sets are denoted by a capital calligraphic letter, 
while their corresponding generating functions with respect to the length
are denoted by the same capital plain letter. 
The bar sign always denotes the intersection of the set with $\D$
(e.g., $\bar{\A} = \A \cap \D$).
Furthermore, for any path $\alpha$, we use the subscript $\alpha$ 
in the notation of a set to denote its subset of paths starting with $\alpha$, 
e.g., $\mathcal{B}_{\alpha} = \{p \in \B: p \text{ starts with } \alpha\}$  
and $B_\alpha$ denotes the corresponding generating function. 
We also use the Iverson notation 
$[P]$, which Proposition $P$ is equal to 1, if $P$ is true, or 0, if $P$ is false. 

There are several different options for the definition of the set
$\A$ of representatives.
In general, we prefer to define $\A$
so that the sequence $(h_i)$ of each representative is
the maximum in its class, because this makes the task of enumeration much easier.
When this is the case, the proof of Proposition \ref{p1} is obvious and it is omitted.
However, when we are also interested in evaluating the number of Dyck classes,
it is more convenient to define the set of representatives so that this sequence
is the minimum in its class (see the cases uuu, uud, duu). 
Then, the set of representatives of Dyck classes
is the set $\bar{\A}$.
Unfortunately, this is not the case for the string udu,
where $\A$ is defined so that this
sequence is the lexicographically minimum in its class
and the representatives of Dyck classes are not always Dyck paths.
This makes this case far more complex than the others.

In sections 2 and 3, we enumerate the $\tau$-equivalence classes of ballot paths
for any string $\tau$ of length 2 and 3 respectively.
In section 4, we enumerate the $\tau$-equivalence classes of Dyck paths
for any string $\tau$ of length 3.

The next two Tables summarize the results of the following sections, 
where the entries in the last column 
refer to the corresponding sequence in \cite{Sloane} (OEIS),
whereas they are left blank whenever this sequence is a new one.
The first three rows of Table \ref{table2}
were given in \cite{BP}.

\begin{table}[ht]
\begin{center}
\begin{tabular}[c]{c||c|c|c|c|c|c|c|c|c|c|c|c|c|}
\hline
$\tau \ \backslash \ n$ & 1 & 2 & 3 & 4 & 5 & 6 & 7 & 8 & 9 & 10 & 11 & 12 & OEIS \\ \hline \hline
du & 1 & 1 & 2 & 3 & 5 & 8 & 13 & 21 & 34 & 55 & 89 & 144 &  A000045 \\ \hline
ud & 1 & 2 & 3 & 5 & 8 & 13 & 21 & 34 & 55 & 89 & 144 & 233 &  du shifted \\ \hline
uu & 1 & 2 & 3 & 5 & 8 & 14 & 24 & 42 & 73 & 128 & 224 & 393 & \\ \hline
dd & 1 & 1 & 1 & 2 & 3 & 5 & 7 & 12 & 18 & 31 & 47 & 81 & A191385 \\ \hline
uuu & 1 & 1 & 2 & 3 & 5 & 8 & 13 & 21 & 34 & 55 & 89 & 144 &  A000045 \\ \hline
ddd & 1 & 1 & 1 & 1 & 1 & 2 & 3 & 5 & 7 & 10 & 13 & 20 &   \\ \hline
uud & 1 & 1 & 2 & 3 & 4 & 6 & 9 & 13 & 19 & 28 & 41 & 60 & A000930  \\ \hline
duu & 1 & 1 & 1 & 2 & 3 & 4 & 6 & 9 & 13 & 19 & 28 & 41 &  uud shifted\\ \hline
udd & 1 & 1 & 1 & 2 & 3 & 4 & 5 & 8 & 12 & 17 & 23 & 35 &   \\ \hline
ddu & 1 & 1 & 1 & 1 & 2 & 3 & 4 & 5 & 8 & 12 & 17 & 23 &  udd shifted \\ \hline
udu & 1 & 1 & 2 & 3 & 5 & 8 & 13 & 21 & 34 & 55 & 89 & 144 &  A000045 \\ \hline
dud & 1 & 1 & 1 & 2 & 3 & 5 & 7 & 11 & 16 & 26 & 39 & 63 &   \\ \hline
\end{tabular}
\caption{Number of equivalence classes of ballot paths of length $n$ for various strings.}
\label{table1}
\end{center}
\end{table}

\begin{table}[ht]
\begin{center}
\begin{tabular}[c]{c||c|c|c|c|c|c|c|c|c|c|c|c|c|}
\hline
$\tau \ \backslash \ n$ & 1 & 2 & 3 & 4 & 5 & 6 & 7 & 8 & 9 & 10 & 11 & 12 & OEIS \\ \hline \hline
uu, dd & 1 & 2 & 4 & 9 & 22 & 56 & 147 & 393 & 1065 & 2915 & 8042 & 22330 & A244886 \\ \hline
ud & 1 & 2 & 5 & 14 & 41 & 121 & 354 & 1021 & 2901 & 8130 & 22513 & 61713 & A244885 \\ \hline
du & 1 & 2 & 5 & 13 & 34 & 89 & 233 & 610 & 1597 & 4181 & 10946 & 28657 &  A001519 \\ \hline
uuu, ddd & 1 & 1 & 2 & 4 & 8 & 17 & 37 & 81 & 180 & 405 & 917 & 2090 &   \\ \hline
uud, udd & 1 & 2 & 4 & 8 & 17 & 35 & 75 & 157 & 337 & 712 & 1529 & 3248 &   \\ \hline
duu, ddu & 1 & 1 & 2 & 4 & 8 & 17 & 35 & 75 & 157 & 337 & 712 & 1529 &  uud shifted\\ \hline
udu, dud & 1 & 2 & 4 & 9 & 22 & 54 & 134 & 335 & 843 & 2132 & 5409 & 13761 &   \\ \hline
\end{tabular}
\caption{Number of equivalence classes of Dyck paths of semilength $n$ for various strings.}
\label{table2}
\end{center}
\end{table}

\section{Strings of length 2}
\subsection{The strings ud and du}
Every $p \in \mathcal{P}$ is uniquely decomposed according to \eqref{clusters} as
$p = a_0 \Us\Ds a_1 \Us\Ds a_2 \cdots \Us\Ds a_k$.

Similarly, every $p$ is uniquely decomposed as
$p = a_0 \Ds\Us a_1 \Ds\Us a_2 \cdots \Ds\Us a_k$.

Let $\A^{(\Us\Ds)}$ (resp. $\A^{(\Ds\Us)}$) be the set of ballot paths, 
the components of which have the form
$a_i = \Us^s, \quad s \geq 0$.

Equivalently, $\A^{(\Us\Ds)}$ (resp. $\A^{(\Ds\Us)}$) is the set of ballot paths avoiding dd
(resp. avoiding dd and ending with u).
Hence, by deleting the last rise of every path in $\A_{n+1}^{(\Ds\Us)}$,
we obtain bijectively every path in $\A_{n}^{(\Us\Ds)}$, so that
\begin{equation}\label{eq1uddu}
|\A_{n+1}^{(\Ds\Us)}| = |\A_{n}^{(\Us\Ds)}|.
\end{equation}

Then, we have the following result:
\begin{Proposition}\label{puddu}
The number of $\Us\Ds$ (resp. $\Ds\Us$)-equivalence
classes of ballot paths of length $n \geq 1$ is equal to $f_{n+1}$
(resp. $f_n$).
\end{Proposition}
\begin{proof}
Clearly, $|\A^{(\Us\Ds)}_1|=1$ and $|\A^{(\Us\Ds)}_2|=2$.
Furthermore, from each path $p \in \A^{(\Us\Ds)}_{n+2}$, 
by deleting its last ud, if $a(p) = \varepsilon$, 
or its last step, if $a(p) \neq \varepsilon$, 
we obtain respectively a path $q_1 \in \A^{(\Us\Ds)}_n$ or a path $q_2 \in \A^{(\Us\Ds)}_{n+1}$.
Since this procedure is clearly reversible, we obtain that
$|\A^{(\Us\Ds)}_{n+2}|= |\A^{(\Us\Ds)}_{n+1}| + |\A^{(\Us\Ds)}_{n}|$, $n \geq 1$. 

The result for the $\Ds\Us$-equivalence classes then follows from relation \eqref{eq1uddu}. 
\end{proof}

\subsection{The string uu}
Every $p \in \mathcal{P}$ is uniquely decomposed according to \eqref{clusters} as
$p = a_0 \Us^{r_1} a_1 \Us^{r_2} a_2 \cdots \Us^{r_k} a_k$, $r_i \geq 2$, $i \in [k]$.
Notice that $a_0 = (\Us\Ds)^t$, $t \geq 0$, 
whereas $a_0$ can also be equal to $(\Us\Ds)^t \Us$, $t \geq0$, for $k=0$.

Let $\A$ be the set of ballot paths, 
the components of which have either one of the following forms:
\[
a_i = \Ds^s, \qquad 
a_i = \Ds^{h_i} (\Us\Ds)^t, \qquad
a_i = \Ds^{h_i-1} (\Us\Ds)^t,
\]
where $i,s,t \geq 1$, and $a_k$ can also be empty.

Equivalently, $\A$ is the set of ballot paths $p$
avoiding $\Ds\Us\Ds\Ds$, avoiding $\Ds\Us\Ds$ at height greater than 1,
and not ending with $\Ds\Us$, unless $p = (\Us\Ds)^{t}\Us$, $t \geq 1$.

\begin{proof}[Proof of Proposition \ref{p1} for $\tau = \Us \Us$.]

Given a ballot path $p = a_0 \Us^{r_1} a_1 \Us^{r_2} a_2 \cdots \Us^{r_k} a_k$, 
setting $h^\prime_0 = 0$, 
$a_0^\prime = a_0$ and, for $i \in [k]$, 
$a_i^\prime = \Ds^{s_i} (\Ds\Us)^{t_i}$
and $h_{i+1}^\prime = h_i^\prime + h(a_i^\prime)$, where
\[
s_i = \begin{cases}
|a_i|, &  |a_i| \leq h_i^\prime, \\
h_i^\prime, & |a_i|-h_i^\prime >0 \text{ is even}, \\
h_i^\prime -1, & |a_i|-h_i^\prime > 0 \text{ is odd}, 
\end{cases}
\]
and $2t_i = |a_i|-s_i$,
we obtain inductively a sequence of paths $a_i^\prime$.
Let
$p^\prime = a_0^\prime \Us^{r_1} a_1^\prime \Us^{r_2} a_2^\prime \cdots \Us^{r_k} a_k^\prime$.
It is easy to check that $p^\prime \in \A$ and $p^\prime \sim p$.
Furthermore, since no two paths of $\A$ are equivalent,
we obtain the required result.
\end{proof}

\begin{Proposition}\label{puu}
The number of $\Us\Us$-equivalence classes of ballot paths of length $n$ is equal to 
the $n$-th coefficient of the generating function
\[
A(x) = \frac{1 - x - x^4}{1 - 2 x + x^3 - x^4 + x^5}.
\]
\end{Proposition}
\begin{proof}
For the enumeration of the set $\A$, we define the sets 
\[
\B = \{p \in \A: p \text{ avoids dud at height greater than zero}\}
\]
and
\[
\G = \{p \in \A: p \text{ avoids dud}\}.
\]

We first consider the following decomposition of a path $p \in \A \setminus \bar{\A}$:
\[
p = (\Us\Ds)^i\Us, i \geq 0, \quad \text{or} \quad 
p = \alpha \Us \beta,
\quad \text{or} \quad 
p = \alpha \Us \Us, \quad \text{or} \quad 
p = \alpha\Us \beta^\prime \Us q,
\]
where $\alpha \in \bar{\A}$, $\beta \in \bar{\B} \setminus \{\varepsilon\}$,
$\beta^\prime \in \bar{\B}$ and $q \in \G \setminus \{\varepsilon, \Us\Ds\Us\}$.
It follows that
\[
A -\bar{A} = \frac{x}{1-x^2} + x\bar{A}(\bar{B}-1) + x^2 \bar{A} + x^2 \bar{A}\bar{B}(G-1-x^3), 
\]
so that
\begin{equation}\label{uuA}
A = \frac{x}{1-x^2} + \bar{A} \left( 1 + x(\bar{B}-1) + x^2 + x^2\bar{B}(G-1-x^3) \right).
\end{equation}

Next, by considering the decomposition 
$\alpha = \beta \Us \gamma \Ds$ of a path
$\alpha  \in \bar{\A} \setminus \{\varepsilon\}$,
where $\beta \in \bar{A}$, $\gamma \in \bar{\B}$ and $\gamma$ does not end with dud,
we obtain that
$\bar{A}-1 = x^2\bar{A}(\bar{B}-x^2(\bar{B}-1))$,
which gives 
\begin{equation}\label{uuAbar} 
\bar{A} = \frac{1}{1-x^2(x^2 + \bar{B}-x^2\bar{B})} 
= \frac{1}{1-x^4- x^2(1-x^2)\bar{B}}, 
\end{equation}
while, by considering the same decomposition for the set 
$\bar{\B} \setminus \{\varepsilon\}$,
where $\beta \in \bar{B}$, $\gamma \in \bar{\G}$,
we obtain that
$\bar{B}-1 = x^2\bar{B}\bar{G}$,
which gives 
\begin{equation}\label{uuBbar}
\bar{B} = \frac{1}{1-x^2\bar{G}}.
\end{equation}

Furthermore, by considering the following decomposition of a path $q \in \G \setminus \bar{\G}$:
\[
q = \Us r, \quad \text{or} \quad 
q = \Us\Ds\Us r, \quad \text{or} \quad 
q = \gamma \Us r^\prime,
\]
where $r \in \G \setminus \{ \Us\Ds\Us \}$, 
$\gamma \in \bar{\G} \setminus \{\varepsilon, \Us\Ds\}$ and $r^\prime \in \G \setminus \{\varepsilon,\Us\Ds\Us\}$,
we deduce that
\[
G - \bar{G} = (x+x^3)(G - x^3) + x(\bar{G} -1- x^2)(G -1- x^3),
\]
which gives
\begin{equation}\label{uuG}
G -1- x^3 = \frac{\bar{G}-1+x}{1-x\bar{G}}.
\end{equation}

Finally, since $\bar{G} = 1+ x^2 M(x^2)$ (see \cite{Sun}),
using relations \eqref{uuA}, \eqref{uuAbar}, \eqref{uuBbar} and \eqref{uuG}, 
after some simple calculations, we obtain the required result.
\end{proof}

\subsection{The string dd}
Every $p \in \mathcal{P}$ is uniquely decomposed according to \eqref{clusters} as
$p = a_0 \Ds^{r_1} a_1 \Ds^{r_2} a_2 \cdots \Ds^{r_k} a_k$, $r_i \geq 2$, $i\in [k]$.

Let $\A$ be the set of ballot paths, the nonempty components of which have the form
$a_i = \Us^s$, $s \geq 1$,
whereas $a_k$ can also be empty.

Equivalently, $\A$ is the set of ballot paths
avoiding $\Us\Ds\Us$ and not ending with $\Us\Ds$.

\begin{Proposition}\label{pdd}
The number of $\Ds\Ds$-equivalence classes of ballot paths of length $n$ is equal to 
the $n$-th coefficient of the generating function
\[
A(x) 
= \frac{1+x^2M(x^2)}{1-x+x^2-x^3M(x^2)}.
\]
\end{Proposition}
\begin{proof}
A path $p \in \A \setminus \bar{\A}$ is decomposed as
$p = \alpha \Us q$, where $\alpha \in \bar{\A}$, $q \in \A$.
Hence,
$A = \bar{A} + x\bar{A}A$,
so that
\begin{equation}\label{ddA}
A = \frac{\bar{A}}{1-x\bar{A}}. 
\end{equation}

Clearly, every path $\beta \in \D$ avoiding udu 
either belongs to $\bar{\A}$ or it has the form
$\beta = \alpha \Us\Ds$, $\alpha \in \bar{\A}$.
Then, since Dyck paths avoiding udu are enumerated by 
$1+x^2M(x^2)$ (see \cite{Sun}), we have that
\begin{equation}\label{ddAbar}
1+x^2M(x^2) = \bar{A} + x^2 \bar{A}.
\end{equation} 

From relations \eqref{ddA} and \eqref{ddAbar}, we obtain the required result.
\end{proof}

\paragraph{Remark:}
The coefficients of $A(x)$ form the sequence A191385 in \cite{Sloane},
where a different combinatorial interpretation is given:
$[x^n]A(x)$ is the number of Motzkin paths of length $n$ with no horizontal steps $\Hs$ at positive height 
and no ascents of length 1.

More generally, the number of $\Ds^\mu$-equivalence classes, $\mu \geq 2$,
of ballot paths of length $n$ is equal to the 
number of 
Motzkin paths of length $n$ with no h at positive
height and no ascents of length less than $\mu$. 
To verify this, we consider the mapping
$\phi$ between the set of ballot paths and the set of Motzkin paths 
with no h at positive height, defined recursively as follows:
\[
\phi (\varepsilon) = \varepsilon, \qquad
\phi (\Us p) = \Hs \phi(p), \qquad
\phi (\Us \alpha \Ds p) = \Us r(\alpha) \Ds \phi(p),
\]
where $p \in \mathcal{P}$, $\alpha \in \D$ and 
$r(\alpha)$ is the reversed path of $\alpha$
(i.e., its symmetrical path with respect to a vertical axis).
It is easy to check that $\phi$ is a length preserving bijection.
Furthermore, since 
$\A^{(\Ds^\mu)}$ is the set of ballot paths avoiding 
$\Us\Ds^i\Us$, $i \in [\mu-1]$, ending with u or $\Ds^\mu$,
we can easily deduce that 
$\phi \left(\A^{(\Ds^\mu)} \right)$ is equal to the set of
Motzkin paths of length $n$ with no h at positive
height and no ascents of length less than $\mu$, giving the required result.

\section{Strings of length 3}
\subsection{The string uuu}\label{uuu}
Every $p \in \mathcal{P}$ is uniquely decomposed according to \eqref{clusters} as
$p = a_0 \Us^{r_1} a_1 \Us^{r_2} a_2 \cdots \Us^{r_k} a_k$, $r_i \geq 3$, $i\in [k]$.

Let $\A$ be the set of ballot paths, with their components $a_i$, $i\in [k]$,
having either one of the following forms:
\[
a_i = \Ds^s, \qquad
a_i = \Ds^{h_i} (\Us\Ds)^t, \qquad 
a_i = \Ds^{h_i-1} (\Us\Ds)^t, 
\]
where $s,t \geq 1$,
while $a_k$ can also be empty and $a_0 \in \{\varepsilon, (\Us\Ds)^t, \Us(\Us\Ds)^t: t \geq 1\}$.

Equivalently, $\A$ is the set of ballot paths $p$
not starting with uudd, avoiding dudd and duud, avoiding dud at height greater than 1,
and ending with uuu or with d, unless $p = \Us$.

\begin{proof}[Proof of Proposition \ref{p1} for $\tau = \Us \Us \Us$.]
Given a ballot path $p = a_0 \Us^{r_1} a_1 \Us^{r_2} a_2 \cdots \Us^{r_k} a_k$, 
setting $h^\prime_0 = 0$, 
$a_0^\prime = (\Us \Ds)^{|a_0|/2}$ if $|a_0|$ is even 
or $a_0^\prime = \Us(\Us \Ds)^{(|a_0|-1)/2}$ if $|a_0|$ is odd,
$a_i^\prime = \Ds^{s_i} (\Us\Ds)^{t_i}$, $i \in [k]$,
and $h_{i+1}^\prime = h_i^\prime + h(a_i^\prime) + r_{i+1}$, $0 \leq i \leq k$, 
where $r_{k+1} = 0$ and, for $i \geq 1$,
\begin{equation*}
s_i = \begin{cases}
|a_i|, &  |a_i| \leq h_i^\prime, \\
h_i^\prime, & |a_i|-h_i^\prime >0 \text{ is even}, \\
h_i^\prime -1, & |a_i|-h_i^\prime > 0 \text{ is odd}, 
\end{cases}
\end{equation*}
and $2t_i = |a_i|-s_i$, 
we obtain inductively a sequence of paths $a_i^\prime$.
Let
$p^\prime = a_0^\prime \Us^{r_1} a_1^\prime \Us^{r_2} a_2^\prime \cdots \Us^{r_k} a_k^\prime$.
It is easy to check that $p^\prime \in \A$ and $p^\prime \sim p$.
Furthermore, since no two paths of $\A$ are equivalent,
we obtain the required result.
\end{proof}

\begin{Proposition}\label{puuu}
The number of $\Us\Us\Us$-equivalence classes of ballot paths of length $n \geq 1$ is equal to 
$f_{n}$.
\end{Proposition}
\begin{proof}
Clearly, $|\A_1|= |\A_2| = 1$. 
Let $\mathcal{S}_n$ (resp. $\mathcal{T}_n$)
be the set of all $p \in \A_n$ with $a(p) = \varepsilon$
(resp. $a(p) \neq \varepsilon$).
We first show that 
\begin{equation}\label{uuuT}
|\mathcal{T}_{n+1}| = |\A_n|, \qquad n \geq 0.
\end{equation}
Indeed, for $p = a_0 \Us^{r_1}a_1  \cdots \Us^{r_k}a_k \in \mathcal{T}_{n+1}$,
so that $a_k \neq \varepsilon$, we define a path $q \in \A_{n+1}$ as follows:
\[
q = \begin{cases}
a_0 \Us^{r_1}a_1  \cdots \Us^{r_k}a_k^\prime, & k>0, \\
a_0^\prime, & k=0,
\end{cases}
\]
where
\[
a_k^\prime = 
\begin{cases}
\Ds^{s-1}, & a_k = \Ds^s, \\
\Ds^{h_k}(\Us\Ds)^{t-1}, & a_k = \Ds^{h_k-1}(\Us\Ds)^{t}, \\
\Ds^{h_k-1}(\Us\Ds)^{t}, & a_k = \Ds^{h_k}(\Us\Ds)^{t}, 
\end{cases}
\qquad
\text{and}
\qquad
a_0^\prime = 
\begin{cases}
\Us (\Us\Ds)^{t-1}, & a_0 = (\Us\Ds)^{t}, \\
(\Us\Ds)^{t}, & a_0 = \Us(\Us\Ds)^{t}. 
\end{cases}
\]
Since the mapping $p \mapsto q$ is clearly a bijection,
we obtain relation \eqref{uuuT}.

Next, we show that
\begin{equation}\label{uuuS1}
|\mathcal{S}_{n+2}| = |\mathcal{S}_{n}| + |\mathcal{T}_{n-2}| + |\mathcal{T}_{n-1}|, 
\qquad n \geq 1.
\end{equation}
Indeed, let $p = a_0 \Us^{r_1}a_1  \cdots \Us^{r_{k-1}}a_{k-1}\Us^{r_k} \in \mathcal{S}_{n+2}$.
If $r_k \geq 5$, then, by deleting the last 2 rises, 
we obtain a path of $\mathcal{S}_{n}$,
whereas, if $r_k=4$ (resp. $r_k=3$), then,
by deleting $\Us^{r_k}$ we obtain a path in $\mathcal{T}_{n-2}$
(resp. $\mathcal{T}_{n-1}$).
Since this procedure is reversible,
we obtain relation \eqref{uuuS1}.

Then, using induction and relations  \eqref{uuuT} and \eqref{uuuS1},
we show that $|\A_{n+2}| = |\A_{n+1}| +|\A_{n}|$.  
Indeed,
\begin{align*}
|\A_{n+2}| &= |\mathcal{T}_{n+2}| + |\mathcal{S}_{n+2}|
= |\mathcal{A}_{n+1}| + |\mathcal{S}_{n}| + |\mathcal{T}_{n-2}| + |\mathcal{T}_{n-1}|
\\&
= |\mathcal{A}_{n+1}| + |\mathcal{S}_{n}| + |\mathcal{A}_{n-3}| + |\mathcal{A}_{n-2}|
= |\mathcal{A}_{n+1}| + |\mathcal{S}_{n}| + |\mathcal{A}_{n-1}|
\\&= |\mathcal{A}_{n+1}| + |\mathcal{S}_{n}| + |\mathcal{T}_{n}|
= |\mathcal{A}_{n+1}| + |\mathcal{A}_{n}| .
\end{align*}
\end{proof}

\subsection{The string ddd}
Every $p \in \mathcal{P}$ is uniquely decomposed according to \eqref{clusters} as
$p = a_0 \Ds^{r_1} a_1 \Ds^{r_2} a_2 \cdots \Ds^{r_k} a_k$, $r_i \geq 3$, $i\in [k]$.

Let $\A$ be the set of ballot paths, the nonempty components of which have the form
$a_i = \Us^s$, $s \geq 1$, whereas $a_k$ can also be empty.

Equivalently, $\A$ is the set of ballot paths
avoiding udu and uddu, and ending with u or ddd.

\begin{Proposition}\label{pddd}
The number of $\Ds\Ds\Ds$-equivalence classes of ballot paths of length $n$ is equal to 
the $n$-th coefficient of the generating function $A = A(x)$, satisfying the relation
\[
x^2(1-2x+x^2-x^4) A^3 + 2x(1-x)^2 A^2 + (1-3x+x^2)A -1=0.
\]
\end{Proposition}
\begin{proof}
A path $p \in \A \setminus \bar{\A}$ is decomposed as
$p = \alpha \Us q$, where $\alpha \in \bar{\A}$, $q \in \A$.
Hence,
$A = \bar{A} + x\bar{A}A$,
so that
\begin{equation}\label{dddA}
\bar{A} = \frac{A}{1+xA}.
\end{equation}
Moreover, a nonempty path $\alpha \in \bar{\A}$ is decomposed as
\[
\alpha = \beta \Us \gamma^\prime \Ds, \quad \text{or} \quad 
\alpha = \beta \Us \gamma \Us \delta \Us \Ds \Ds \Ds,
\qquad 
\beta, \gamma, \delta \in \bar{\A}, \gamma^\prime \in \bar{\A} \setminus \{ \varepsilon\}.
\]
Hence, 
\begin{equation}\label{dddA2}
\bar{A} -1 = x^2\bar{A}(\bar{A}-1) + x^6 \bar{A}^3.
\end{equation}

From relations \eqref{dddA} and \eqref{dddA2}, we obtain the required result.
\end{proof}

\subsection{The strings uud and duu}\label{uud}

Every $p \in  \mathcal{P}$ is uniquely decomposed according to \eqref{clusters} as
$p = a_0 \Us\Us\Ds a_1 \Us\Us\Ds a_2 \cdots \Us\Us\Ds a_k$.

Similarly, every $p$ is uniquely decomposed as
$p = a_0 \Ds\Us\Us a_1 \Ds\Us\Us a_2 \cdots \Ds\Us\Us a_k$.

Let $\A^{(\Us\Us\Ds)}$ be the set of ballot paths 
with their components $a_i$ having either one of the forms:
\[
a_i = \Ds^s, s \leq h_i-1, \qquad
a_i = \Ds^{h_i} (\Us \Ds)^t, \qquad
a_i = \Ds^{h_i} (\Us \Ds)^t \Us, 
\]
where $s,t \geq 0$, 
and let $\A^{(\Ds\Us\Us)}$ be the set of ballot paths 
with their components $a_i$ having either one of the forms:
\[
a_0 = (\Us \Ds)^t \Us, \qquad a_0 = (\Us \Ds)^t \Us^2, k >0, \qquad a_0 = (\Us \Ds)^t, k=0,
\] 
\[
a_i = \Ds^s, s \leq h_i-2, \qquad
a_i = \Ds^{h_i-1} (\Us \Ds)^t, \qquad
a_i = \Ds^{h_i-1} (\Us \Ds)^t \Us, 
\]
where $0<i<k$ and $s,t \geq 0$,
\[
a_k = \Ds^s, s \leq h_k-1, \qquad
a_k = \Ds^{h_k} (\Us \Ds)^t, \qquad
a_k = \Ds^{h_k} (\Us \Ds)^t \Us, 
\]
where $k>0$ and $s,t \geq 0$.

We will give the proof of Proposition \ref{p1} for the string duu only,
since the proof for uud it is similar and slightly easier.

\begin{proof}[Proof of Proposition \ref{p1} for $\tau = \Ds \Us \Us$.]
Given a ballot path $p = a_0 \Ds\Us\Us a_1 \Ds\Us\Us a_2 \cdots \Ds\Us\Us a_k$, 
setting $h^\prime_0 = 0$, 
$a_0^\prime = (\Us\Ds)^{(|a_0|-1)/2}\Us$, if $|a_0|$ is odd, 
$a_0^\prime = (\Us\Ds)^{|a_0|/2}\Us$, if $|a_0|$ is even and $k=0$,
and
$a_0^\prime = (\Us\Ds)^{(|a_0|-2)/2}\Us^2$, if $|a_0|$ is even and $k>0$,
and $h_{i+1}^\prime = h_i^\prime + h(a_i^\prime)+[i<k]$, $i \in [k]$, 
and also setting
\begin{equation*}
a_i^\prime =
\begin{cases}
\Ds^{|a_i|}, & |a_i| \leq h_i^\prime -1, \\
\Ds^{h_i^\prime-1} (\Us\Ds)^{(|a_i|-h_i^\prime)/2}\Us, & |a_i|-h_i^\prime > 0 \text{ is even}, \\
\Ds^{h_i^\prime-1} (\Us\Ds)^{(|a_i|-h_i^\prime+1)/2}, & |a_i|-h_i^\prime > 0 \text{ is odd}, 
\end{cases}
\qquad 1 \leq i < k,
\end{equation*}
and
\[
a_k^\prime =
\begin{cases}
\Ds^{|a_k|}, &  |a_k| \leq h_k^\prime, \\
\Ds^{h_k^\prime} (\Us\Ds)^{(|a_k|-h_k^\prime)/2}, &  |a_k|-h_k^\prime > 0 \text{ is even}, \\
\Ds^{h_k^\prime} (\Us\Ds)^{(|a_k|-h_k^\prime-1)/2}\Us, &  |a_k|-h_k^\prime > 0 \text{ is odd}, 
\end{cases}
\qquad k > 0,
\]
we obtain inductively a sequence of paths $a_i^\prime$.
Let $p^\prime = a_0^\prime \Ds\Us\Us a_1^\prime \Ds\Us\Us a_2^\prime \cdots \Ds\Us\Us a_k^\prime$.
It is easy to check that $p^\prime \in \A^{(\Ds\Us\Us)}$ and 
$p^\prime \underset{\Ds\Us\Us}{\sim} p$.
Furthermore, since no two paths of $\A^{(\Ds\Us\Us)}$ are equivalent,
we obtain the required result.
\end{proof}

We note that 
\[
|\A_1^{(\Us\Us\Ds)}| = |\A_2^{(\Us\Us\Ds)}| =1
\qquad \text{and} \qquad |\A_3^{(\Us\Us\Ds)}| = 2,
\]
\[
|\A_1^{(\Ds\Us\Us)}| = |\A_2^{(\Ds\Us\Us)}| = |\A_3^{(\Ds\Us\Us)}| = 1,
\]
and
\begin{equation}\label{eq2}
|\A_{n+1}^{(\Ds\Us\Us)}| = |\A_n^{(\Us\Us\Ds)}|, \qquad n \geq 3.
\end{equation}

For the proof of the last equality, consider the bijection
which maps each 
\[
p = a_0 \Ds\Us\Us a_1 \Ds\Us\Us a_2 \cdots \Ds\Us\Us a_{k-1} \Ds\Us\Us a_k \in \A^{(\Ds\Us\Us)}
\]
to the path
\[
q = a_0^\prime \Us\Ds\Ds a_1 \Us\Ds\Ds a_2 \cdots \Us\Ds\Ds a_{k-1} \Us\Ds\Ds a_k^\prime 
\in \A^{(\Us\Us\Ds)},
\]
where $a_0^\prime$ is obtained by deleting the last step of $a_0$ and, for $k>0$,
\[
a_k^\prime =
\begin{cases}
a_k, &  a_k = \Ds^s, 0\leq s \leq  h_k-1, \\
\Ds^{h_k-1} (\Us\Ds)^{t+1}, &  a_k= \Ds^{h_k} (\Us\Ds)^{t}\Us, t \geq 0, \\
\Ds^{h_k-1} (\Us\Ds)^{t}\Us, &  a_k = \Ds^h_k (\Us\Ds)^t, t \geq 0. 
\end{cases}
\]

\begin{Proposition}\label{pduu}
The number $|\A_n^{(\tau)}|$ of $\tau$-equivalence classes of ballot paths of length $n \geq 1$ 
satisfies the recurrence relation 
\[
|\A_{n+3}^{(\tau)}| = |\A_{n+2}^{(\tau)}| + |\A_{n}^{(\tau)}|,
\]
where $\tau \in \{\Us\Us\Ds, \Ds\Us\Us\}$.
\end{Proposition}
\begin{proof}
From each path $p \in \A_{n+3}^{(\Us\Us\Ds)}$, by deleting its last uud, if $a(p) = \varepsilon$, 
or its last step, if $a(p) \neq \varepsilon$, we obtain respectively
a path $q \in \A_{n}^{(\Us\Us\Ds)}$ or a path $q \in \A_{n+2}^{(\Us\Us\Ds)}$.
Since this procedure is clearly reversible, we obtain the required relation 
for $\tau = \Us\Us\Ds$.

The result for the string $\Ds\Us\Us$ then follows from relation \eqref{eq2}.
\end{proof}

\subsection{The strings udd and ddu}
Every $p \in \mathcal{P}$ is uniquely decomposed according to \eqref{clusters} as
$p = a_0 \Us\Ds\Ds a_1 \Us\Ds\Ds a_2 \cdots \Us\Ds\Ds a_k$.

Similarly, every $p$ is uniquely decomposed as
$p = a_0 \Ds\Ds\Us a_1 \Ds\Ds\Us a_2 \cdots \Ds\Ds\Us a_k$.

Let $\A^{(\Us\Ds\Ds)}$ (resp. $\A^{(\Ds\Ds\Us)}$) be the set of ballot paths, 
the $\Us\Ds\Ds$-components (resp. $\Ds\Ds\Us$-components) of which have the form
$a_i = \Us^s$, $s \geq 0$.

Equivalently, $\A^{(\Us\Ds\Ds)}$ (resp. $\A^{(\Ds\Ds\Us)}$) 
is the set of ballot paths avoiding udu and ddd, and ending with udd or u
(resp. with u).

Hence, by deleting the last rise of every path in $\A_{n+1}^{(\Ds\Ds\Us)}$,
we obtain bijectively every path in $\A_n^{(\Us\Ds\Ds)}$, so that
\begin{equation}\label{eq3}
|\A_{n+1}^{(\Ds\Ds\Us)}| = |\A_n^{(\Us\Ds\Ds)}|, \qquad n \geq 0.
\end{equation}

\begin{Proposition}\label{pudd}
The number 
of $\Us\Ds\Ds$ (resp. $\Ds\Ds\Us$)-equivalence classes 
of ballot paths of length $n \geq 1$ is equal to
the $n$-th (resp. $(n-1)$-th) coefficient of the generating function
\[
\frac{C(x^4)}{1-xC(x^4)},
\]
where
\[
[x^n]\frac{C(x^4)}{1-xC(x^4)} 
= \sum_{i=0}^{\lfloor n/4 \rfloor} \frac{n-4i+1}{n-3i+1} \binom{n-2i}{i}.
\]
\end{Proposition}
\begin{proof}
First assume that $\A = \A^{(\Us\Ds\Ds)}$.
A path $p \in \A \setminus \bar{\A}$ is decomposed as
$p = \alpha \Us q$, where $\alpha \in \bar{\A}$, $q \in \A$.
Hence,
$A = \bar{A} + x\bar{A}A$,
so that
\[ 
A = \frac{\bar{A}}{1-x\bar{A}}. 
\] 

Moreover, a nonempty path $\alpha \in \bar{\A}$ is decomposed as
$\alpha = \beta \Us \gamma \Us \Ds \Ds$, where $\beta, \gamma \in \bar{\A}$.
Hence, 
$\bar{A} -1 = x^4\bar{A}^2$,
which implies that
$\bar{A} = C(x^4)$.

It follows that 
\[
A(x) = \frac{C(x^4)}{1-xC(x^4)}.
\]

Then by expanding $A(x)$ to a geometric series and using the well known formula
\begin{equation}\label{CatalanPowers}
[x^n]C^s(x) = \frac{s}{2n+s} \binom{2n+s}{n}, \qquad n \geq 0, s >0,
\end{equation}
we can easily obtain the formula for $|\A_n^{(\Us\Ds\Ds)}|$.
The formula for $|\A_n^{(\Ds\Ds\Us)}|$ then follows from relation \eqref{eq3}.
\end{proof}

\subsection{The string udu}\label{udu}
Every $p \in \mathcal{P}$ is uniquely decomposed according to \eqref{clusters} as
$p = a_0 (\Us\Ds)^{r_1}\Us a_1 (\Us\Ds)^{r_2}\Us a_2 \cdots (\Us\Ds)^{r_k}\Us a_k$, $r_i \geq 1$, $i \in [k]$.

Let $\A^{(\Us\Ds\Us)}$ be the set of ballot paths, the components of which have the form
$a_i = \Us^s$, $s \geq 0$.

Equivalently, $\A^{(\Us\Ds\Us)}$ is the set of ballot paths avoiding dd and ending with u.

It follows that $\A^{(\Us\Ds\Us)} = \A^{(\Ds\Us)}$, so that by Proposition \ref{puddu}
we have the following result:
\begin{Proposition}\label{pudu}
The number of $\Us\Ds\Us$-equivalence classes of ballot paths of length $n \geq 1$ 
is equal to $f_n$. 
\end{Proposition}

\subsection{The string dud}
Every $p \in \mathcal{P}$ is uniquely decomposed according to \eqref{clusters} as
$p = a_0 \Ds(\Us\Ds)^{r_1} a_1 \Ds(\Us\Ds)^{r_2} a_2 \cdots \Ds(\Us\Ds)^{r_k}\Us a_k$, $r_i \geq 1$, $i \in [k]$.

Let $\A$ be the set of ballot paths, the nonempty components of which have either 
one of the following forms:
\[
a_i = \Us^s , \quad s \geq 1 + [0 < i < k], \qquad \qquad
a_i = \Ds, \quad 0 < i < k. 
\]

Equivalently, $\A$ is the set of ballot paths $p$
avoiding $\Ds^4$, $\Us^2\Ds^2$, $\Ds^2\Us^2$, $\Us^2\Ds\Us^2$, 
starting with uu or udud and ending with uu, dud or dudu,
unless $p \in \{ \varepsilon, \Us\}$.

\begin{Proposition}\label{pdud}
The number of $\Ds\Us\Ds$-equivalence classes of ballot paths of length $n$ 
is equal to the $n$-th coefficient of the generating function $A = A(x)$, 
satisfying the relation
\[
x^2(1-x-x^2)A^3 + 2x(1-\frac{3}{2}x-x^2+x^3+x^4-x^5)A^2 +
(1-3x-x^2+3x^3+x^4-3x^5)A - 1+x^2-x^4=0.
\] 
\end{Proposition}
\begin{proof}
A path $p \in \A \setminus \bar{\A}$ is decomposed as
$p = \alpha \Us q$, where $\alpha \in \bar{\A}$, $q \in \A$.
Hence,
$A = \bar{A} + x\bar{A}A$,
so that
\begin{equation}\label{dudA}
\bar{A} = \frac{A}{1+xA}.
\end{equation}

We set $\bar{\B}$ to be the set of Dyck paths $\beta$
avoiding $\Ds^4$, $\Us^2\Ds^2$, $\Ds^2\Us^2$, $\Us^2\Ds\Us^2$
and starting with uu or udud, unless $\beta = \Us\Ds$.
Then, since every nonempty path $\alpha \in \bar{\A}$ 
is uniquely decomposed as
$\alpha = \beta \Us\Ds$, $\beta \in \bar{\B}$, we have that
\begin{equation}\label{dudAB}
\bar{A} - 1 = x^2 \bar{B}.
\end{equation}

A path $\beta \in \bar{\B}$ is uniquely decomposed as
\[
\beta = \gamma \Us\Ds, \quad \text{or} \quad
\beta =  \alpha \Us \beta^\prime \Us\Ds \Ds, \quad \text{or} \quad
\beta = \alpha \Us \alpha^\prime \Us \beta^\prime \Us\Ds \Ds \Ds, 
\qquad
\alpha, \alpha^\prime \in \bar{\A} ,
\gamma \in \bar{\B} \cup \{\varepsilon\},
\beta^\prime \in \bar{\B}.
\]
It follows that
\[
\bar{B} = x^2 + x^2\bar{B} + x^4\bar{A}\bar{B} + x^6\bar{A}^2\bar{B},
\]
which, combined with \eqref{dudAB}, gives
\[
\bar{A}-1 = x^4 + x^2(\bar{A}-1) +x^4\bar{A}(\bar{A}-1)+x^6\bar{A}^2(\bar{A}-1),
\]
so that
\begin{equation}\label{dudA2}
x^6\bar{A}^3  + x^4(1-x^2)\bar{A}^2 -(1-x^2+x^4)\bar{A} + 1-x^2+x^4 =0.
\end{equation}

From relations \eqref{dudA} and \eqref{dudA2}, we obtain the required result.
\end{proof}


\section{Enumeration of Dyck classes}

In this section, we evaluate the number of $\tau$-equivalence classes of Dyck paths of semilength $n$, for any string $\tau$ of length 3. 
By symmetry, it is enough to deal only with the strings uuu, uud, duu and udu.
For the first 3 cases, we use the set $\A^{(\tau)}$ of representatives 
defined for the corresponding $\tau$'s in the previous section, 
whereas for the case of udu, we define a new one.

In the sequel, all generating functions used for the enumeration of the Dyck classes
are defined with respect to the semilength of the Dyck paths.

\begin{Proposition}\label{pADyck}
The set $\bar{\A}^{(\tau)} = \A^{(\tau)} \cap \D$ is equal to the set
of representatives of the Dyck classes, for $\tau \in \{ \Us\Us\Us, \Us\Us\Ds, \Ds\Us\Us\}$.
\end{Proposition}
\begin{proof}
It is enough to show that if $p \in \mathcal{P}$, $p^\prime \in \A^{(\tau)}$
and $p \underset{\tau}{\sim} p^\prime$, then $h(p^\prime) \leq h(p)$,
since then, if $p \in \D$, it follows that $h(p^\prime) = h(p) = 0$, so that $p^\prime \in \D$,
which gives the required result.

We restrict ourselves to the string $\tau = \Us\Us\Us$, 
since the proofs for the other cases is similar.

Let $p = a_0 \Us^{r_1} a_1 \Us^{r_2} a_2 \cdots \Us^{r_k} a_k \in \mathcal{P}$ and 
$p^\prime = a_0^\prime \Us^{r_1} a_1^\prime \Us^{r_2} a_2^\prime \cdots \Us^{r_k} a_{k}^\prime \in \A^{(\Us\Us\Us)}$, decomposed according to \eqref{clusters},
such that $p \underset{\tau}{\sim} p^\prime$.
We will show by induction that $h_i^\prime \leq h_i$, for all $i\leq k+1$,
which in particular gives that $h(p^\prime) \leq h(p)$.

For $0\leq i \leq k$, we have that
\begin{align*}
h^\prime_{i+1} = h^\prime_i + h(a_i^\prime) + r_{i+1}
&= r_{i+1} + 
\begin{cases}
h_i^\prime - |a_i|, & |a_i| \leq h_i^\prime \\
0, & |a_i| - h_i^\prime >0 \text{ is even}\\
1, & |a_i| - h_i^\prime >0 \text{ is odd}\\
\end{cases}
\\&
\leq  r_{i+1} + 
\begin{cases}
h_i + h(a_i), & |a_i| \leq h_i^\prime \\
0, & |a_i| - h_i^\prime >0 \text{ is even}\\
1, & |a_i| - h_i^\prime >0 \text{ is odd}\\
\end{cases}
\\& 
\leq h_i + h(a_i) + r_{i+1} = h_{i+1},
\end{align*}
because, if $h_{i+1}^\prime = 1 + r_{i+1}$ and since $h_{i+1}, h_{i+1}^\prime$
have the same parity, we must have that $h_{i+1} \geq 1 + r_{i+1}$.
\end{proof}

\subsection{The string uuu}
In this section we use the set $\A$ of representatives defined in section \ref{uuu}


\begin{Proposition}\label{puuu_Dyck}
The number of $\Us\Us\Us$-equivalence classes of Dyck paths of semilength $n$ 
is equal to the $n$-coefficient of the generating function
\[
F(x) = \frac{1+xG }{1 - x(G-1)^2},
\]
where $G = G(x)$ satisfies the relation
\[
xG^3 - (1+2x)G^2 + (1+3x)G - x = 0.
\]
\end{Proposition}
\begin{proof}
In view of Proposition \ref{pADyck}, for the enumeration
of the classes of Dyck paths, it is enough to enumerate the set
$\bar{\A}$ of Dyck paths not starting with uudd, and avoiding dudd, duud,
and dud at height greater than 1.
  
For this, we consider the generating function
$F = F(x) = \sum_{p \in \bar{\A}} x^{|p|_\Us}$.
We also define the sets 
\[
\B = \{ \alpha \in \D: \alpha \text{ avoids dudd, duud, and dud at height greater than 0}\}
\]
and
\[
\G = \{ \alpha \in \D: \alpha \text{ avoids dud, duud}\}.
\]

Using the last return decomposition of a nonempty Dyck path $\alpha$:
\[
\alpha = \Us \beta \Ds, 
\qquad \text{or} \qquad
\alpha = \alpha^\prime \Us \beta^\prime \Ds,
\qquad \alpha^\prime, \beta, \beta^\prime \in \D, \alpha^\prime \neq \varepsilon,
\]
we have that:

\begin{itemize}
\item[i)] $\alpha \in \bar{\A} \setminus \{\varepsilon\}$ iff
$\alpha^\prime \in \bar{\A} \setminus \{\varepsilon\}$, $\beta \in \B$,
$\beta^\prime \in \B_{\Us\Us} \cup \{\varepsilon\}$ and $\beta, \beta^\prime$ do not end with ud, 
which gives that
\[
F-1 = x(B - xB) + x(F-1)(1+ B_{\Us\Us} - xB_{\Us\Us}),
\]
so that
\begin{equation}\label{FuuuDyck}
F-1 = \frac{xB}{1-xB_{\Us\Us}}.
\end{equation}

\item[ii)]
$\alpha \in \B \setminus \{\varepsilon\}$ iff
$\alpha^\prime \in \B \setminus \{\varepsilon\}$, $\beta \in \G$ and
$\beta^\prime \in \G_{\Us\Us}\cup \{\varepsilon\}$, 
which gives 
\[
B-1 = xG + x(B-1) (1+ G_{\Us\Us}),
\]
so that
\begin{equation}\label{BuuuDyck}
B-1 
= \frac{xG}{1-x(1 + G_{\Us\Us})}.
\end{equation}

\item[iii)]
$\alpha \in \B_{\Us\Us}$ iff
$\alpha^\prime \in \B_{\Us\Us}$, $\beta \in \G \setminus \{ \varepsilon\}$ and
$\beta^\prime \in \G_{\Us\Us}\cup \{\varepsilon\}$, 
which gives 
\[
B_{\Us\Us} = x(G-1)+x B_{\Us\Us} (1+ G_{\Us\Us}),
\]
so that
\begin{equation}\label{B2uuuDyck}
B_{\Us\Us} 
= \frac{x(G -1)}{1-x(1 + G_{\Us\Us})}.
\end{equation}

\item[iv)]
$\alpha \in \G \setminus \{ \varepsilon\}$ iff
$\alpha^\prime \in \G \setminus \{ \varepsilon\}$, $\beta \in \G$ and
$\beta^\prime \in \G_{\Us\Us}$, 
which gives
\[
G-1 = x G + x (G-1)G_{\Us\Us},
\]
so that
\begin{equation}\label{GuuuDyck}
G -1 
= \frac{xG}{1-xG_{\Us\Us}}.
\end{equation}

\item[v)]
$\alpha \in \G_{\Us\Us}$ iff
$\alpha^\prime \in \G_{\Us\Us}$, $\beta \in \G \setminus \{\varepsilon\}$ and
$\beta^\prime \in \G_{\Us\Us}$, 
which gives 
\[
G_{\Us\Us} = x(G -1) + x G_{\Us\Us}^2,
\]
so that
\begin{equation}\label{G2uuuDyck}
G_{\Us\Us} 
= \frac{x(G-1)}{1-xG_{\Us\Us}}.
\end{equation}
\end{itemize}

From relations \eqref{GuuuDyck} and \eqref{G2uuuDyck}, it follows that 
\[
G_{\Us\Us} = (G-1)^2 / G \qquad \text{and} \qquad
x(G -1)^3 - (G - 1)G + xG^2=0.
\]
Thus, substituting in relations \eqref{BuuuDyck} and \eqref{B2uuuDyck}, we obtain
\[
B_{\Us\Us} = (G-1)^2 \qquad \text{and} \qquad B = (G-1)^2 + G = B_{\Us\Us} + G.
\]
Finally, substituting in relation \eqref{FuuuDyck}, 
we obtain the required result.
\end{proof}

\subsection{The strings uud and duu}

In this section, we use the sets of representatives $\A^{(\Us\Us\Ds)}$ and 
$\A^{(\Ds\Us\Us)}$ defined in section \ref{uud}. 
Then, we have that
\begin{equation}\label{eq4}
|\bar{\A}_{2n+2}^{(\Ds\Us\Us)}| = |\bar{\A}_{2n}^{(\Us\Us\Ds)}|, \qquad n \geq 0.
\end{equation}

For the proof of the above equality, consider the bijection
\[
p\mapsto q, \qquad
p = a_0 \Ds\Us\Us a_1 \Ds\Us\Us a_2 \cdots \Ds\Us\Us a_k \in \bar{\A}^{(\Ds\Us\Us)},
q = a_0^\prime \Us\Us\Ds a_1 \Us\Us\Ds a_2 \cdots \Us\Us\Ds a_k^\prime \in \bar{\A}^{(\Us\Us\Ds)},
\]
where $a_0^\prime$ is obtained by deleting the last ud of $a_0$, if $k=0$,
whereas, if $k>0$, 
$a_0^\prime$ is obtained by deleting the last u of $a_0$, and
\[
a_k^\prime =
\begin{cases}
\Ds^{s-1}, & a_k = \Ds^s, 2 \leq s \leq h_k, \\
\Ds^{h_k-1} (\Us\Ds)^t, & a_k = \Ds^{h_k}(\Us\Ds)^t, t \geq 1.
\end{cases}
\]

\begin{Proposition}\label{puud_duu_Dyck}
The number
of $\Us\Us\Ds$ (resp. $\Ds\Us\Us$)-equivalence classes of Dyck paths of semilength 
$n \geq 1$ is equal to 
the $n$-th (resp. $(n-1)$-th) coefficient of the generating function
\[
\frac{1}{1 - x(1+x)C(x^2)},
\]
where
\[
[x^n]\frac{1}{1 - x(1+x)C(x^2)} = 
\sum_{i=0}^{\lfloor (n-1)/2 \rfloor} \sum_{j=0}^{\lfloor (n-2i)/2 \rfloor}
\frac{n-2i-j}{n-j} \binom{n-i}{j} \binom{n-2i-j}{j}.
\]
\end{Proposition}
\begin{proof}
Let $F = F(x) = \sum_{p \in \bar{\A}} x^{|\alpha|_\Us}$.

Using the definition of $\A^{(\Us\Us\Ds)}$ in section \ref{uud},
we can easily check that $\bar{\A}^{(\Us\Us\Ds)}$
is the set of Dyck paths avoiding both uuu and dud at height greater than 0. 
It follows that
every nonempty path $\alpha \in \bar{\A}^{(\Us\Us\Ds)}$ is decomposed as 
$p = \Us \beta \Ds \gamma$,
where $\gamma \in \bar{\A}^{(\Us\Us\Ds)}$ and 
$\beta$ belongs to the set $\B$ of Dyck paths avoiding uuu and dud. 
Hence
\begin{equation}\label{eq1uudDyck}
F -1 = xBF.
\end{equation}

Clearly, $\B = \{\varepsilon\} \cup \B_{\Us\Ds} \cup \B_{\Us\Us\Ds}$, so that
\begin{equation}\label{eq2uudDyck}
B = 1 + B_{\Us\Ds} + B_{\Us\Us\Ds}.
\end{equation}

A path $\beta \in \B_{\Us\Ds}$ 
is decomposed as $\beta = \Us \Ds \beta_1$, 
where $\beta_1 \in \B_{\Us\Us\Ds} \cup \{\varepsilon\}$.  
Hence, 
\begin{equation}\label{eq3uudDyck}
B_{\Us\Ds} = x(1+ B_{\Us\Us\Ds}). 
\end{equation}

On the other hand, a path $\beta \in \B_{\Us\Us\Ds}$ 
is decomposed as $\beta = \Us\Us \Ds \beta_1 \Ds \beta_2$, where $\beta_1, \beta_2 \in \B_{\Us\Us\Ds} \cup \{\varepsilon\}$.  
Hence, 
\[
B_{\Us\Us\Ds} = x^2 (1+B_{\Us\Us\Ds})^2.
\] 
This shows that $B_{\Us\Us\Ds} = C(x^2)-1$ and, using relations
\eqref{eq1uudDyck}, \eqref{eq2uudDyck} and \eqref{eq3uudDyck},
we obtain that
\[
F(x) = \frac{1}{1-x(1+x)C(x^2)}.
\]

Then, expanding $F(x)$ to a geometric series and using formula
\eqref{CatalanPowers}, we obtain the required formula for $|\bar{\A}_{2n}^{(\Us\Us\Ds)}|$.
The formula for $|\bar{\A}_{2n}^{(\Ds\Us\Us)}|$ then follows from relation
\eqref{eq4}.
\end{proof}

\subsection{The string udu}

The set of representatives used in section \ref{udu}
does not contain any Dyck paths, and therefore it is not convenient
for the present case. 
For this reason, we will consider a more suitable set of representatives,
as follows:

We define $\A$ to be the set of ballot paths
$p = a_0 (\Us\Ds)^{r_1} \Us a_1 (\Us\Ds)^{r_2}\Us a_2 \cdots (\Us\Ds)^{r_k}\Us a_k$,
$r_i \geq 1, i \in [k]$, 
the nonempty components of which have either one of the following forms:
\[
a_0 = \Us^s\Ds^s, \quad  a_0 = \Us^s\Ds^{s-1}, s \geq 1, \qquad \text{if } k =0,
\]
\[
a_0 = \Us^s, s \in [3], \quad  a_0 =  \Us^s\Ds^{s} \text{ or } \Us^{s+1}\Ds^{s}, s \geq 2, 
\qquad \text{if } k >0,
\]
and, for $i \in [k]$,
\[
a_i = 
\Us, i < k, \quad  
a_i = \Us^2, h_i=1, i<k, \quad  
a_i = \Ds^s, s \geq 1 + [i<k], \quad  
a_i =  \Us^s\Ds^{s+h_i} \text{ or } \Us^s\Ds^{s+h_i-1}, s \geq 1. 
\]

\begin{proof}[Proof of Proposition \ref{p1} for $\tau = \Us\Ds\Us$.]
Given a Dyck path 
$p = a_0 (\Us\Ds)^{r_1} \Us a_1 (\Us\Ds)^{r_2}\Us a_2 \cdots (\Us\Ds)^{r_k}\Us a_k$,
setting $h_0^\prime = 0$, 
$a_i^\prime = \Us^{s_i}\Ds^{t_i}$,
$h_{i+1}^\prime = h_i^\prime + h(a_i^\prime) +[i<k]$, where
\[
s_0 = 
\begin{cases}
|a_0|, & k>0, |a_0| \leq 3\\
\lceil |a_0|/2 \rceil, & \text{otherwise}, 
\end{cases}
\]
\[
s_i = 
\left( \left\lceil \dfrac{|a_i|-h_i^\prime }{2}\right\rceil \right)^+ 
+ [i<k]\left([|a_i|=2][h_i^\prime=1] + [|a_i|=1]\right),
\qquad  i \in [k]
\]
and
\[
t_i = |a_i|-s_i, \qquad i \geq 0,
\]
we obtain inductively a sequence of paths $a_i$.
Let
$p^\prime = a_0^\prime (\Us\Ds)^{r_1}\Us a_1^\prime (\Us\Ds)^{r_2}\Us a_2^\prime \cdots 
(\Us\Ds)^{r_k}\Us a_k^\prime$.
It is easy to check that $p^\prime \in \A$ and $p^\prime \sim p$.
Furthermore, since no two paths of $\A$ are equivalent,
we obtain the required result.
\end{proof}

In the next result, we characterize the elements of $\A$ which are
representatives of Dyck classes.

\begin{Proposition}\label{p2Dyck}
A path $p^\prime \in \mathcal{A}$ is equivalent to a Dyck path
$p$ iff it has either one of the following forms:
\begin{enumerate}
\item $p^\prime$ is a Dyck path or 
\item
$p^\prime = \alpha (\Us\Ds)^r \Us^2 \beta $, where $r \geq 1$, $\alpha, \beta \in \D$, 
$\beta$ starts with $\Us^2$ and 
$\alpha$ ends with $\Ds^2$ and it has either an occurrence of $\Ds^3$ or an occurrence
of $\Ds^2$ before an occurrence of $\Us^2\Ds^2$.  
\end{enumerate}
\end{Proposition}
\begin{proof}
Using the same notation as in the proof of Proposition \ref{p1},
we consider 
\[
p = a_0 (\Us\Ds)^{r_1} \Us a_1 (\Us\Ds)^{r_2}\Us a_2 \cdots (\Us\Ds)^{r_k}\Us a_k
\text{ and }
p^\prime = a_0^\prime (\Us\Ds)^{r_1}\Us a_1^\prime (\Us\Ds)^{r_2}\Us a_2^\prime \cdots 
(\Us\Ds)^{r_k}\Us a_k^\prime,
\]
with $p \in \D$, $p^\prime \in \A$ and $p \sim p^\prime$.

We will show that either $p^\prime \in \D$, or $p^\prime$ is of the form 
$\alpha (\Us\Ds)^r \Us^2 \beta$.

Clearly, if $k =0$, then $p^\prime = (\Us\Ds)^{|a_0|/2} \in \D$, so that
we may assume that $k>0$.

Let $t$ be the greatest element of $[k]$ such that $h_t^\prime = 1$,
if such an element exists, or $t=0$, otherwise.
Clearly, if $t=k$, since $h_{k+1}, h_{k+1}^\prime$ have the same parity and $h_{k+1}=0$,
it follows that $h_{k+1}^\prime =0$, so that $p^\prime \in \D$.

In the sequel, we assume that $t<k$. We first show the following result:

\begin{Lemma*}
If there exists $i \in [t+1,k]$, such that $h_i^\prime \leq h_i$, then $p^\prime \in \D$.
\end{Lemma*}
For this, we show by induction that $h_j^\prime \leq h_j$ for every $j \in [i,k+1]$.
Let $j \in [i,k]$ with $h_j^\prime \leq h_j$. We will show that $h_{j+1}^\prime \leq h_{j+1}$.

If $|a_j|=1$, then $a_j^\prime = a_j$ and the result follows easily.

On the other hand, if $|a_j| \neq 1$, then 
$|a_j^\prime|_\Us = \left(\lceil (|a_j| - h_j^\prime -1)/2 \rceil \right)^+$.
We can easily check that
\[
\frac{|a_j|-h_j \prime }{2} \leq |a_j|_\Us + \frac{h_j-h_j^\prime}{2},
\]
which, since $(h_j-h_j^\prime)/2$ is a nonnegative integer, gives that
\[
|a_j^\prime|_\Us \leq |a_j|_\Us + \frac{h_j-h_j^\prime}{2}.
\]
Then using relation \eqref{parity}, we deduce that
\[
h_{j+1}- h_{j+1}^\prime \geq h_j - h_j^\prime 
+ 2 \left( |a_j|_\Us - |a_j|_\Us - \frac{h_j - h_j^\prime}{2} \right)= 0. 
\]

This shows that, $h_j^\prime\leq h_j$, for every $j \in [i,k+1]$. In particular, 
$h_{k+1}^\prime \leq h_{k+1} = 0$, 
so that $p^\prime \in \D$, completing the proof of the lemma.

If $t=0$, then, since $h_1^\prime \leq h_1$, by the lemma it follows that
$p^\prime \in \D$, so that in the sequel we may assume that $t>0$.

If now $h_i^\prime = 2$, for some $i \in [t+1,k]$, then 
by the lemma it follows that $p^\prime \in \D$;
so that we can restrict ourselves to the case where $h_i^\prime \geq 3$ for every $i \in [t+1,k]$.
It follows that $|a_t| = 1$ or $|a_t|=2$. 

If $a_t = \Us$, or $a_t = \Us^2$, or $a_t = \Ds^2$ with $h_t \geq 5$, 
we can easily check that $h_{t+1}^\prime \leq h_{t+1}$, 
so that by the lemma it follows that $p^\prime \in \D$.

We finally consider the remaining case where $a_t = \Ds^2$ and $h_t=3$. 
Then, since $a_t^\prime = \Us^2$, $h_{t+1}^\prime = 4$, $h_{t+1} = 2$,
applying relation \eqref{parity} for every $i \in [t+1, k-1]$, we obtain that

\begin{align*}
h_k - h_k^\prime &= 
h_{t+1} - h_{t+1}^\prime + 2\sum_{i=t+1}^{k-1} (|a_i|_\Us - |a_i^\prime|_\Us)
=-2 + 2\sum_{\substack{i\in [t+1,k-1]\\ |a_i| \geq 2}} |a_i|_\Us,
\end{align*}
since $|a_i^\prime|_\Us = 0$, for every $i\in [t+1,k-1]$ with $|a_i| \geq 2$.

We consider two subcases. 

If $|a_i|_\Us \geq 1$, for some $i \in [t+1,k-1]$, then $h_k - h_k^\prime \geq 0$, 
so that by the lemma we obtain that $p^\prime \in \D$.

If, on the other hand, $|a_i|_\Us = 0$ for every $i \in [t+1,k-1]$, 
we have that $h_k - h_k^\prime = -2$.
If $p^\prime \notin \D$, then $|a_k^\prime|_\Ds = |a_k|$, so that
\begin{align*}
2 \leq h_{k+1}^\prime &= h_k^\prime - |a_k| = h_k +2 - |a_k|
\leq h_k - h(a_k) + 2 = h_{k+1} + 2 = 2.
\end{align*}
This shows that $h_{k+1}^\prime = 2$, that is $p^\prime$ ends at height 2.

Furthermore, since $h_i^\prime \geq 3$ for every $i \in [t+1,k]$, we obtain that
the subpath $\beta$ of $p^\prime$ starting from the second rise of $a_t^\prime$ 
and ending at the end of $p^\prime$ is a Dyck path starting with $\Us^2$.

If we set 
$a = a_0^\prime (\Us\Ds)^{r_1}\Us a_1^\prime (\Us\Ds)^{r_2}\Us a_2^\prime \cdots 
(\Us\Ds)^{r_{t-1}}\Us a_{t-1}^\prime$ 
and $r = r_t \geq 1$, we obtain that
$\alpha$ is a Dyck path ending with $\Ds^2$ and 
$p^\prime = \alpha (\Us\Ds)^r \Us^2 \beta$.

In the sequel, assuming that $\alpha$ avoids $\Ds^3$, we will show that $\alpha$ has an occurrence of $\Ds^2$ followed by 
an occurrence of $\Us^2\Ds^2$.

Firstly, we note that, since in this case $h_t=3$, we obtain
\begin{equation}\label{eqp2}
\sum_{i=0}^{t-1} (h(a_i) - h(a_i^\prime)) = 2.
\end{equation}

Clearly, since $\alpha$ avoids $\Ds^3$, each $a_i^\prime$, $i \geq 0$, 
has one of the following forms:
\[
a_0^\prime = \Us^s, s \in \{0,1,2,3\},  \qquad \text{or} \qquad
a_0^\prime = \Us^s \Ds^2, s \in \{2,3\}
\]
and, for $i \in [t-1]$,
\[
a_i^\prime = \Us^s, \qquad \text{or} \qquad a_i^\prime = \Us^s \Ds^2, \qquad s \in \{0,1,2\}.
\]
It is easy to check that in each case $h(a_i) - h(a_i^\prime) \in \{-4,-2,0,4\}$.
Then, in view of relation \eqref{eqp2}, it follows that there exists
$i \in [t-1]$ such that $h(a_i) - h(a_i^\prime) = -2$.
This can occur only when $a_i^\prime = \Us^s \Ds^2$
and $a_i = \Us^{s-1} \Ds^3$ or $a_i = \Ds^3 \Us^{s-1} $, 
where $s \in \{1,2\}$, which ensures the existence of $\Us^2 \Ds^2$ in $\alpha$.

If there is no occurrence of $\Ds^2$ on the left of $a_i^\prime$ in $\alpha$, 
since $h_i^\prime \in \{1,2\}$, 
we can easily check that $i \in \{1,2\}$ and that $a_j = a_j^\prime$ for all $j<i$,
so that $h_i = h_i^\prime$.
Then, 
since the last point of $a_i^\prime$ in $p^\prime$ lies at height 0 or 1
and $h(a_i) - h(a_i^\prime) = -2$,
it follows that the last point of $a_i$ in $p$ lies at height -2 or -1,
which is a contradiction since $p \in \D$.

For the converse, it is enough to show that if 
$p^\prime = \alpha (\Us\Ds)^r \Us^2 \beta \in \A$
with the above properties, then there exists $p \in \D$ such that $p \sim p^\prime$.

Firstly, assuming that $\alpha$ has an occurrence of $\Ds^3$, 
we set $\gamma$ to be the path obtained from $\alpha$,
by changing the first fall of the leftmost such occurrence into a rise.

On the other hand, if $\alpha$ avoids $\Ds^3$ and has an occurrence of $\Us^2\Ds^2$ and an occurrence of $\Ds^2$ on the left of it, 
we set $\gamma$ to be the path obtained from $\alpha$ by changing
$\Us^2\Ds^2$ into $\Us\Ds^3$ and the leftmost $\Ds^2$ into $\Us^2$.

In both cases, we can easily check that $\gamma \in \mathcal{P}$, $\alpha \sim \gamma$ and $\gamma$ ends at height 2.
Now, since $\beta$ starts with $\Us^2$, there exist 
$\beta_1,\beta_2 \in \D$ with $\beta_1 \neq \varepsilon$
such that $\beta = \Us \beta_1 \Ds \beta_2$.
If we set $p = \gamma (\Us\Ds)^r\Us \Ds^2 \beta_1 \Ds \beta_2$, it is easy to see that 
$p \in \D$ and $p \sim p^\prime$.
\end{proof}

\begin{Proposition}\label{uduDyck}
The number of $\Us\Ds\Us$-equivalence classes of Dyck paths of semilength $n$ 
is equal to the $n$-th coefficient of the generating function
\[
F(x) = \frac{x(1 - x)^2 (1 + G   +  xG) + x^5(1 +  xG)G^2}{(1 - x)((1 - x)^2 + (x -  2)x^2 G)}
- \frac{x^4(1 -  x +   x^3) (1  +   xG)GT}{(1  -  x)^2(1  -  x  +  x^3  -  xT)}
\]
where the generating functions $G=G(x)$ and $T = T(x)$ are given by
\[
G = x(1+G) +x^2(1+xG)(1 + x(1+xG))G
\]
and
\[
x^3 T^3 + x^2 T^2 + (x-1)T +x^2 = 0.
\]

\end{Proposition}
\begin{proof}
In order to enumerate the $\Us\Ds\Us$-equivalence classes of $\D$, 
according to Proposition \ref{p2Dyck},
it is enough to 
evaluate the generating function $F = F(x)$ 
of the set of paths $p^\prime$ having either one of the two forms described in 
Proposition \ref{p2Dyck}. 

For the first form, we note that
$\bar{\A}$ is the set of Dyck paths avoiding
$\Ds^2\Us^2$, $\Ds^2\Us\Ds^2$, $\Us^4$ at height greater than 0,
$\Us^2\Ds^s \Us$, $s \geq 2$, at height greater than 1,
and not ending with $\Ds^2 \Us \Ds$.

Let $\B$ (resp. $\G$) be the set of paths in $\bar{\A} \setminus \{\varepsilon\}$
also avoiding $\Us^4$ at height 0 and
$\Us^2\Ds^s \Us$, $s \geq 2$, at height 1 (resp. $\Us^2\Ds^2$),
and let $\K$ (resp. $\Ll$) be the set of paths in $\bar{\A}$ (resp. $\G$)
ending with $\Ds^2$.
We denote by $H, B, G, K, L$ the generating functions, 
with respect to the semilength, of the sets 
$\bar{\A}, \B, \G, \K, \Ll$ respectively.

Clearly, a path $p \in \bar{\A} \setminus \{\varepsilon\}$ (resp. $\G$)
is decomposed as
\[
p = (\Us\Ds)^s, 
\quad \text{or} \quad
p = \delta,
\quad \text{or} \quad
p = \delta (\Us\Ds)^{s+1},
\qquad
s \geq 1,  
\]
where $\delta \in \K$ (resp. $\delta \in \Ll$).

It follows that,
$H-1 = \dfrac{x}{1-x} + K + \dfrac{x^2}{1-x}K$,
which gives
\begin{equation}\label{K} 
K = \dfrac{-1+(1-x)H}{1-x+x^2}. 
\end{equation}
Similarly, we obtain that
\begin{equation}\label{L}
L = \frac{G - x -x G}{1-x+x^2}.
\end{equation}

A path $p \in \K$ is decomposed as 
$p = \alpha \Us \beta \Ds$
where $\alpha \Us\Ds \in \bar{\A} \setminus (\K \cup \{\varepsilon\})$
and $\beta \in \B$,
so that $K = (H-1-K)B$,
which combined with \eqref{K} gives
\begin{equation}\label{eqA}
H = 1+ xH +x(1+x(H-1))B.
\end{equation}

Next, given a path $p \in \B \setminus \G$, 
we consider the last $\Us^2 \Ds^s$ decomposition
$p = \beta \gamma$,
where $\beta \in \B$ ends with $\Us^2 \Ds^s$, $s \geq 2$, and
$\gamma \in \G_{\Us\Ds\Us} \cup \{ \varepsilon\}$.

Clearly, by deleting the last peak (i.e., $\Us\Ds$) of $\beta$, 
we obtain a path $\beta^\prime \in \B$.
Since the mapping $\beta \mapsto \beta^\prime$ is clearly a bijection, 
we obtain that the set of paths $\beta$ is enumerated
by $xB$ and hence
\begin{equation}\label{eqB}
B - G = xB (1+xG).
\end{equation}

Moreover, if a path $p \in \G$ is decomposed as
$p = \gamma \Us \gamma_1 \Ds$, we have that
$\gamma \Us\Ds \in \G \setminus \Ll$, $\gamma_1 \in \G$
and starts with $\Us\Ds\Us$ or $\Us\Ds^2\Us$.
If $\gamma_1$ starts with $\Us\Ds^2\Us$, then it is decomposed as
$\gamma_1 = \Us \gamma_2 \Ds \gamma_3$, where $\gamma_2 \in \G_{\Us\Ds\Us}$ 
and $\gamma_3 \in \G_{\Us\Ds\Us} \cup \{\varepsilon\}$.
 
It follows that 
$L = (G - L) \left( x G + x x G (1+xG)\right)$,
which combined with relation \eqref{L} gives
\begin{equation}\label{eqG}
G = x(1+G) +x^2(1+xG)(1 + x(1+xG))G. 
\end{equation}

Finally, from relations \eqref{eqA}, \eqref{eqG}, \eqref{eqB}, we deduce that 

\begin{equation}\label{eqA2}
H = \frac{1-x + x(1-2x)G)}{(1-x)^2 + (x-2)x^2G}.
\end{equation}

Next, we deal with the second form of paths 
$p^\prime = \alpha (\Us\Ds)^r \Us \beta$ described in Proposition \ref{p2Dyck}.

For the paths $\alpha$, we define the set
$\E$ of Dyck paths with no occurrence of $\Ds^3$ 
and with no occurrence of $\Ds^2$ followed by an occurrence of $\Us^2\Ds^2$
and we denote by $H_1$, $B_1$, $G_1$
the generating functions of the sets
$\E \cap \K$, $B \cap \E \cap \K$, $\G \cap \E \cap \K$
respectively.

Then, we proceed to the enumeration of the set $\K \setminus \E$ consisting of
the above mentioned paths $\alpha$, with corresponding generating function
$K - H_1$.

Every path $p \in \E \cap \K$ is decomposed as
\[
p = (\Us\Ds)^s\Us \beta \Ds, 
\qquad \text{or} \qquad
p = \delta (\Us\Ds)^{s+1} \Us \gamma \Ds,
\]
where $s \geq 0$,
$\beta \in \left(\B \cap \E\right) \setminus \K$,
$\delta \in \E \cap \K$, $\gamma \in \G \cap \E \setminus (\K \cup \{\Us\Ds\})$.

Moreover, since every path $\beta \in \left(\B \cap \E\right) \setminus \K$
(resp. $\gamma \in \left(\G \cap \E\right) \setminus (\K \cup \{\Us\Ds\})$)
has either one of the forms
\[
(\Us\Ds)^t \qquad \text{or} \qquad \sigma (\Us\Ds)^{t+1},
\]
where $t \geq 1$ and $\sigma \in \B \cap \E \cap \K$
(resp. $\sigma \in \G \cap \E \cap \K$),
we obtain that the set $\left(\B \cap \E\right) \setminus \K$
(resp. $\left(\G \cap \E\right) \setminus \K$) 
is enumerated by $\frac{x}{1-x}(1+xB_1)$ (resp. $\frac{x}{1-x}(1+xG_1)$).

Thus, we obtain that 
\[
H_1 = \frac{x}{1-x} \frac{x(1+xB_1)}{1-x} + H_1 \frac{x^2}{1-x} 
\left(\frac{x(1+xG_1)}{1-x} -x\right),
\]
which gives 
\begin{equation}\label{eqA_1} 
H_1 = \frac{x^2\left( 1+xB_1 + x^2 H_1 (1+G_1) \right)}{(1-x)^2}.
\end{equation}

Similarly, every path $p \in \G \cap \E \cap \K$, is decomposed as
\[
p = (\Us\Ds)^s\Us \gamma \Ds, 
\qquad \text{or} \qquad
p = \delta (\Us\Ds)^{s+1} \Us \gamma \Ds,
\]
where $s \geq 0$,
$\delta \in \G \cap \E \cap \K$,
$\gamma \in \left(\G_{\Us\Ds\Us} \cap \E\right) \setminus \K$ or
$\gamma = \Us \gamma_1 \Ds \gamma_2 \Ds$, 
$\gamma_1, \gamma_2 \in \left( \G_{\Us\Ds\Us} \cap \E\right) \setminus \K$.

As before, we find that the set $\left( \G_{\Us\Ds\Us} \cap \E\right) \setminus \K$ 
is enumerated by $\frac{x^2}{1-x}(1+xG_1)$.

Hence,
\begin{align}\label{eqG_1} 
G_1  &= \left( \frac{1}{1-x} + \frac{x}{1-x}G_1 \right) x 
\left( \frac{x^2}{1-x}(1+xG_1) + \frac{x^5}{(1-x)^2}(1+xG_1)^2\right) \nonumber
\\&
=\frac{x^3}{(1-x)^2}(1+xG_1)^2 \left( 1+ \frac{x^3}{1-x}(1+xG_1)\right).
\end{align}

Every path $p \in (\B \setminus \G) \cap \E \cap \K$ is decomposed as
\[
p = (\Us\Ds)^s \Us^2\Ds^2 \gamma, 
\qquad s \geq 0, \quad \gamma \in \left(\G_{\Us\Ds\Us} \cap \E \cap \K\right) \cup \{\varepsilon\}.
\]

Hence,  
\begin{equation}\label{eqB_1} 
B_1 - G_1  = \frac{x^2}{1-x}(1+xG_1).
\end{equation}

From relations \eqref{eqA_1}, \eqref{eqG_1} and \eqref{eqB_1}, 
setting $T = B_1 - G_1$, 
we obtain
\begin{equation}\label{eqD_1} 
x^3 T^3 + x^2 T^2 + (x-1)T +x^2 = 0
\end{equation}
and
\begin{equation}\label{eqAD_1} 
H_1 = \frac{(1-x+x^3) T}{(1-x)(1-x+x^3-xT)}.
\end{equation}

Finally, the paths $\beta$ of the second form described in Proposition \ref{uduDyck} 
are decomposed as
\[
\beta = \Us^2 \Ds \gamma_1 \Ds \gamma_2, \qquad
\gamma_1 \in \G, \gamma_2 \in \G_{\Us\Ds\Us}\cup \{ \varepsilon \}, 
\]
and hence they are enumerated by the generating function 
$x^2 G (1+xG)$.

In conclusion, the generating function $F$ is given by the equality
\[
F = H + \frac{x^2}{1-x} \left( K - H_1 \right) x^2 G (1+xG),
\]
which, combined with relations \eqref{K}, \eqref{eqA2} and \eqref{eqAD_1}, 
gives the required result.
\end{proof}


\begin{thebibliography}{99}


\bibitem{BP}
J.-L. Baril and A. Petrossian, Equivalence classes of Dyck paths modulo some statistics, {\em Discrete Math.\/} {\bf 338} (2015), 655--660.

\bibitem{BP2}
J.-L. Baril and A. Petrossian, Equivalence classes of Motzkin paths modulo a pattern of length at most two, 
\emph{J. Integer Sequences} {\bf 18} (2015), Article 15.7.1.


\bibitem{D3}
E. Deutsch, Dyck path enumeration, {\em Discrete Math.\/} {\bf 204} (1999),
167--202. 


\bibitem{MSTT1} 
K. Manes, A. Sapounakis, I. Tasoulas and P. Tsikouras, Counting strings at height $j$ in Dyck paths, {\em J. Statist. Plann. and Infer.} {\bf 141} (2011), 2100--2107.

\bibitem{MSTT2} K. Manes, A. Sapounakis, I. Tasoulas and P. Tsikouras, General results on the enumeration of strings in Dyck paths, {\em Electron. J. Combin.} {\bf 18} (2011), \#P74.


\bibitem{M1}
T. Mansour, Counting peaks at height $k$ in a Dyck path, {\em J. Integer Seq.\/}
{\bf 5} (2002), Article 02.1.1. 

\bibitem{M2}
T. Mansour, Statistics on Dyck paths, {\em J. Integer Seq.\/} {\bf 9} (2006), 
Article 06.1.5. 

\bibitem{MSV}
D. Merlini, R. Sprungoli and M. Verri, Some statistics on Dyck paths, {\em J.
Statist. Plann. and Infer.\/} {\bf 101} (2002), 211--227.

\bibitem{STT1} 
A. Sapounakis, I. Tasoulas and P. Tsikouras, Counting strings in
Dyck paths, {\em Discrete Math.\/} {\bf 307} (2007), 2909--2924.

\bibitem{Sloane} N. J. A. Sloane, {\em The Online Encyclopedia of Integer
Sequences\/} (2013), published electronically at http://oeis.org/.

\bibitem{Sun}
Y. Sun, The statistic ``number of udu's" in Dyck paths, {\it Discrete Math.\/}
{\bf 287} (2004), 177--186. 

\end{thebibliography}
\end{document}